\documentclass[12pt,bezier]{article}
\usepackage{graphicx,color}
\usepackage{epsfig}
\usepackage{amsmath,amsthm,graphicx,amsfonts}

\newcommand{\comment}[1]{}

\newtheorem{theorem}{Theorem}[section]

\newtheorem{lemma}[theorem]{Lemma}

\newtheorem{guess}[theorem]{Conjecture}
\theoremstyle{definition}

\def \cB {{\cal B}}
\def \cC {{\cal C}}
\def \cD {{\cal D}}

\def \cS {{\cal S}}

\def \Z {\mathbb Z}

\author{
S.\ Akbari \footnotemark[1] \footnotemark[2]
\and
A.C.\ Burgess\footnotemark[3]
\and
P.\ Danziger\footnotemark[4]
\and
E.\ Mendelsohn\footnotemark[4] \footnotemark[5]
}

\date{}
\title{Zero-Sum Flows for Steiner Triple Systems}

\begin{document}

\maketitle

\footnotetext[1]{Department of Mathematical Sciences, Sharif University of Technology, Tehran, Iran.}
\footnotetext[2]{School of Mathematics, Institute for Research in Fundamental Sciences, IPM, Tehran, Iran}
\footnotetext[3]{Department of Mathematical Sciences, University of New Brunswick, Saint John, NB, Canada.}
\footnotetext[4]{Department of Mathematics, Ryerson University, Toronto, ON, Canada.}
\footnotetext[5]{Department of Mathematics, University of Toronto, Toronto, ON, Canada.}

\begin{abstract}
Given a $2$-$(v,k,\lambda)$ design, $\cS=(X,\cB)$, a {\it zero-sum $n$-flow} of $\cS$ is a map $f: \cB \longrightarrow \{\pm 1, \ldots ,\pm (n-1)\}$ such that for any point $x\in X$, the sum of $f$ around all the blocks incident with $x$ is zero. 
It has been conjectured that every Steiner triple system, STS$(v)$, on $v$ points $(v>7)$ admits a zero-sum $3$-flow.
We show that for every pair $(v,\lambda)$, for which a triple system, TS$(v,\lambda)$ exists, there exists one which has a zero-sum $3$-flow, except when $(v,\lambda)\in\{(3,1), (4,2), (6,2), (7,1)\}$ and except possibly when $v \equiv 10\pmod{12}$ and $\lambda = 2$.
We also give a $O(\lambda^2v^2)$ bound on $n$ and a recursive result which shows that every STS$(v)$ with a zero-sum $3$-flow can be embedded in an STS$(2v+1)$ with a zero-sum $3$-flow if $v\equiv 3 \pmod 4$, a zero-sum $4$-flow if $v\equiv 3 \pmod 6$ and with a zero-sum $5$-flow if $v\equiv 1 \pmod 4$.
\end{abstract}

\section{ Introduction}

Let $G$ be a graph with the vertex set $V(G)$ and edge set $E(G)$.  A {\it  $k$-edge colouring} of a graph
$G$ is a function $f: E(G)\longrightarrow L$ such that $|L| = k$ and $f(e_1)\neq f(e_2)$ for every two adjacent edges $e_1$ and $e_2$.
The {\it chromatic index} of $G$, denoted by $\chi'(G)$ is the minimum number $k$ for which $G$ has a
$k$-edge colouring. A {\it $1$-factor} of a graph $G$ is a set of independent edges which
covers all vertices of $G$. Given a colouring of $G$, a {\it rainbow $1$-factor} is a $1$-factor all of whose edges have different colours.
We denote the complete graph of order $n$ by $K_n$.

A {\it zero-sum flow} of a graph $G$ is an assignment of non-zero real numbers to the edges of $G$ such that the sum of the values of all edges incident with any given vertex is zero. 
Let $n$ be a natural number. A {\it zero-sum $n$-flow} is a zero-sum flow with values from the set \mbox{$\{\pm 1,\ldots ,\pm(n-1)\}$}. For a subset $S\subseteq E(G)$, the {\it weight of $S$} is defined to be the sum of the values of all edges of $S$. Such flows have been studied in \cite{akb 2013, akb 2012, akb 2010, akb 2009, wang}.
Zero-sum flows are motivated by nowhere-zero flows, which were first introduced by Tutte in 1949 \cite{tutte}.

Let $n$ be a positive integer.  A {\it Latin square of order $n$} with entries from $X$ is an $n \times n$ array $L$ such that every row and column of $L$ is a permutation of $X$. Suppose that $L_1$ is a Latin square of order $n$ with entries from $X$ and $L_2$ is a Latin square of order $n$ with entries from $Y$. We say that $L_1$ and $L_2$ are {\it orthogonal} provided that, for every $x \in X$ and for every $y \in Y$, there is a unique cell $(i, j)$ such that $L_1(i, j) = x$ and $L_2(i, j) = y$. 
A Latin square, $L$, is called {\em indempotent} if $L(i,i)=i$ for every $i$.
It is well known that for every positive integer $n\neq 2, 6$, there exist two orthogonal Latin squares of order $n$. A {\em transversal} of a Latin square is a set of cells which between them contain each row, column and entry exactly once. If a square $L$ has an orthogonal mate, $L'$, it is possible to partition the cells of $L$ into transversals, $T_i$, each of which correspond to the positions of the entry $i$ in $L'$.
We refer the reader to \cite{handbook} for notation and further results on Latin squares.

A \emph{$2$-$(v, k, \lambda)$ design} $\cD$ (briefly, $2$-design), is a pair $(X, \cB)$, where $X$ is a $v$-set of points and $\cB$ is a collection of $k$-subsets of $X$, called {\em blocks}, with the property that every $2$-subset of $X$ is contained in exactly $\lambda$ blocks. Traditionally, the number of blocks and the frequency of occurrences of points in blocks are denoted by $b$ and $r$, respectively. Given an indexing of the points and blocks of a $2$-design, the {\em incidence matrix} of $\cD$ is the $v \times b$ $(0,1)$-matrix $A=[a_{ij}]$, where 
\[
a_{ij} = \left\{\begin{array}{cl}
1 & x_i \in B_j, \, B_j \in \cB \\
0 & \mbox{ otherwise.}
\end{array}
\right.
\]
We refer the reader to \cite{handbook} for notation and further results on designs.

Given a $2$-design $\cD=(X, \cB)$, we define a zero-sum $n$-flow of $\cD$ to be a function $f:\cB \longrightarrow \{\pm 1, \ldots ,\pm (n-1)\}$ such that the sum of the block weights around any point is zero, i.e.\
\[
w(x) = \sum_{x\in B} f(B) = 0.
\]
This is equivalent to finding a vector 
in the nullspace of the incidence matrix of the design whose entries are all in the set $\{\pm 1, \ldots ,\pm (n-1)\}$. Zero-sum $n$-flow of designs have been previously been studied in \cite{akb}.

A {\em parallel class} of a $2$-design is a collection of disjoint blocks which between them contain every point of $X$ exactly once.
A $2$-design is called \emph{resolvable} if there exists a partition of the set of blocks into parallel classes.  
An {\em $\alpha$-resolution class} is a collection of blocks $\cS\subseteq \cB$ which contain every point of $X$ exactly $\alpha$ times. If the block set $\cB$ can be partitioned into $\alpha$-resolution classes we call the design {\em $\alpha$-resolvable}; in this case we denote the number of $\alpha$-resolution classes by $\rho=r/\alpha$.

An {\em automorphism} of a $2$-design is a permutation of the point set which maps blocks to blocks. 
The set of automorphisms of a $2$-design forms a group, called the {\em automorphism group}, $G$, of the design.
A $2$-$(v,k,\lambda)$ design, $(X,\cB)$, is called {\em cyclic} if its automorphism group contains a cycle of length $v$. In such cases there exists a set of {\em starter blocks} $\cS \subseteq \cB$, such that the full design may be obtained by acting on these blocks by $G$.
The orbit of a starter block which has length less than $v$ is called a {\em short orbit}; if the orbit is of length $v$, it is called a {\em full orbit}.

A $2$-$(v,3,1)$ design is called a \emph{Steiner triple system} of order $v$, denoted $\emph{\emph{STS}}(v)$, and a resolvable STS$(v)$ is called a {\em Kirkman triple system}, and denoted KTS$(v)$. 
It is well known that a Steiner triple system on $v$ points exists if and only if $v \equiv 1$ or $3\pmod{6}$ and that a KTS$(v)$ exists if and only if $v\equiv 3\pmod 6$.  
More generally, we call a $2$-$(v,3,\lambda)$ design a {\em triple system } and denote it by $\mathrm{TS}(v,\lambda)$.  We refer the reader to \cite{handbook,col} for further results on triple systems. We have the following existence result.
\begin{theorem} [\cite{hanani}]
A TS$(v, \lambda)$ exists if and only if $v \neq 2$ and $\lambda \equiv 0 \pmod{\gcd{(v - 2, 6)}}$.
\end{theorem}

A connection between zero sum flows on Steiner triple systems and more traditional geometric flows may be obtained by considering the well known embedding of the points of a TS$(7,2)$ into the torus, so that the blocks form equilateral triangles which may be coloured black and white, in such a way that no triangles of the same colour share an edge \cite{emch}. Giving each black triangle weight $1$ and each white triangle weight $-1$, it is easy to see that this is equivalent to a zero-sum $2$-flow of the TS$(7,2)$.

In this paper we consider the application of zero-sum flows to triple systems. 
In \cite{akb} it was conjectured that every $\mathrm{STS}(v)$ with $v>7$ admits a zero-sum $3$-flow. 
We generalise this conjecture to the case of arbitrary $\lambda$ as follows.
\begin{guess} \label{flow conjecture}
Every TS$(v, \lambda)$ admits a zero-sum $3$-flow, except when $(v,\lambda)\in\{(3,1), (4,2), (6,2), (7,1)\}$.
\end{guess}
In support of the original conjecture, a computer search has shown that every STS$(v)$, $7<v\leq 15$, admits a zero-sum $3$-flow; see \cite{akb}. In this paper, as further evidence in support of both the original conjecture in~\cite{akb} and Conjecture~\ref{flow conjecture}, we show that for every admissible order $v>7$ there exists an STS$(v)$ which admits a zero-sum $3$-flow. Further, we show in support of Conjecture~\ref{flow conjecture} that for every admissible $(v,\lambda)\not\in\{(3,1), (4,2), (6,2), (7,1)\}$, there exists a TS$(v,\lambda)$ that admits a zero-sum $3$-flow, except possibly when $\lambda=2$ and $v\equiv 10\pmod{12}$. In this latter case we show that a TS$(v,\lambda)$ admitting a zero-sum $5$-flow exists.

In the rest of this section we give some general results which apply to $2$-designs in general. In the next section we then show the existence result for triple systems mentioned above. In Section~\ref{bound} we give an $O((\lambda v)^2)$ bound for a zero-sum flow on a TS$(v,\lambda)$. Finally, in Section~\ref{recursive}, we provide some recursive $(2v+1)$-constructions.

We now develop some general tools for determining zero-sum flows applicable to $2$-designs.
We begin with the following result on $\alpha$-resolvable designs, recalling that we denote the number of $\alpha$-resolution classes by $\rho=r/\alpha$.
\begin{lemma}
\label{alpha}
An $\alpha$-resolvable $2$-$(v, k, \lambda)$ design has a zero-sum $3$-flow if $\rho > 1$ is odd and a zero-sum $2$-flow if $\rho$ is even.
\end{lemma}\begin{proof}
If $\rho$ is even, we colour all of the blocks of each of the $\alpha$-resolution classes alternately $+1$ and $-1$. Each pair of oppositely signed classes then generates a zero-sum flow. If $\rho$ is odd we choose three classes and colour the blocks of two of them with $+1$ and the blocks of the third $-2$.  The remaining classes (of which there are an even number) have their blocks coloured with $+1$ and $-1$ as above.
\end{proof}
Taking $\alpha=1$ we obtain the following result.
\begin{theorem}
\label{resolvable}
A resolvable design with at least two classes has a zero-sum $3$-flow.
\end{theorem}

Clearly there can be no resolvable $\mathrm{STS}(6v+1)$. A {\em Hanani triple system} of order $6v+1$ has $3v$ almost parallel classes $P_1, P_2, \ldots, P_{3v}$ (each of which contain $6v$ points), and one partial parallel class $P_0$ of size $v$, containing $3v$ points. Thus, Hanani triple systems are in some sense as close to resolvable as one can get when the number of points is $1\pmod 6$. There exists a Hanani triple system of order $6v+1$ if and only if $v \neq 1,2$~\cite{hts}.  We now show that these also admit a zero-sum $3$-flow.

\begin{theorem}
Let $v \geq 3$ be an integer. Any Hanani triple system of order $6v+1$ admits a zero-sum $3$-flow.
\end{theorem}

\begin{proof}
Let the blocks of $P_0$ be $B_1, B_2, \ldots, B_v$. For $i=1,2,\ldots,3v$, let $x_i$ be the point which does not appear in a block of $P_i$.  Then the set of vertices which appear in blocks in $P_0$ is $\{x_1, x_2, \ldots, x_{3v}\}$.  

First suppose that $v$ is even.  For $i=1,2,\ldots,v$, assign $B_i$ the label $(-1)^i$.  Next we colour the blocks of the remaining parallel classes.  For each $i=1,2,\ldots,3v$, if $x_i$ is in a block with label $j$, assign label $j$ to the blocks of parallel class $P_i$.  It is easy to check that the resulting labelling is a zero-sum flow.

Now, suppose that $v$ is odd.  For $i=1,2,\ldots,v-3$, assign $(-1)^i$ to block $B_i$.  Assign $-1$ to $B_{v-2}$ and $B_{v-1}$ and 2 to $B_v$.  Again, for each $i=1,2,\ldots,3v$, if $x_i$ is in a block with label $j$, assign label $j$ to parallel class $P_i$.
\end{proof}

Given a cyclic $2$-$(v,k,\lambda)$ design, we note that the set of blocks developed from a single starter which generates a full orbit forms a $k$-resolution class. Thus, if a cyclic design has no short orbits the design is $k$-resolvable. This observation, in conjunction with Theorem~\ref{resolvable}, gives us the following result.
\begin{theorem}
\label{cyclic}
A cyclic $2$-$(v,k,\lambda)$ design with no short orbits has a zero-sum $3$-flow.
\end{theorem}

Lemma II 2.3 of \cite{beth} states that the incidence matrix of a non-trivial symmetric design is non-singular. We thus have the following result.
\begin{theorem}
\label{symmetric}
A symmetric $2$-design has no zero-sum flow.
\end{theorem}

\section{Existence of zero-sum flows for triple systems}

In this section we show that for every pair $(v,\lambda)$ such that a TS$(v,\lambda)$ exists, there is one with a zero-sum $3$-flow, except when $(v,\lambda)\in\{(3,1), (4,2), (6,2), (7,1)\}$ and except possibly when $v \equiv 10\pmod{12}$ and $\lambda = 2$.
Necessary and sufficient conditions for the existence of a TS$(v,\lambda)$ were settled by Hanani \cite{hanani} in 1961.
\begin{theorem}[\cite{hanani}]
\label{nec}
The necessary conditions for the existence of a TS$(v,\lambda)$ are:
\begin{enumerate}
\item
$\lambda \equiv 1, 5 \pmod{6}$ and $v \equiv 1, 3 \pmod{6}$;
\item
$\lambda \equiv 2, 4 \pmod{6}$ and $v \equiv 0, 1 \pmod{3}$;
\item
$\lambda \equiv 3 \pmod{6}$ and $v \equiv 1 \pmod{2}$;
\item
$\lambda \equiv 0 \pmod{6}$ and $v \geq 3$;
\end{enumerate}
\end{theorem}
\noindent
The $(v,\lambda)$ values in Theorem~\ref{nec} are called {\em admissible}.

We first give a non-existence result as a direct corollary of Theorem~\ref{symmetric}.
\begin{lemma}
\label{non}
A zero-sum flow does not exist for TS$(3,1)$, TS$(4,2)$ and TS$(7, 1)$.
\end{lemma}

Next we give zero-sum flows for some small parameter values.
\begin{lemma}\rule{0ex}{1ex}\\[-4ex]
\label{small}
\begin{enumerate}
\item
There is no TS$(6,2)$ admitting a zero-sum $3$-flow, but there is one with a zero-sum $4$-flow.
\item
There exists a TS$(6,4)$ with a zero-sum $2$-flow 
\item
There exists a TS$(6,6)$ with a zero-sum $3$-flow. 
\end{enumerate}
\end{lemma}
\begin{proof}\rule{0ex}{1ex}
\begin{enumerate}
\item
A computer search showed that the unique TS$(6,2)$ has no zero-sum $3$-flow. A zero-sum $4$-flow is shown below.
$$ 
\begin{array}{ccccccccccc}
B & 123 & 134 & 145 & 156 & 126 & 235 & 346 & 245 & 356 & 246 \\
f(B) & 1 & -2 & 2 & -3 & 2 & -2 & 1 & 1 & 2 & -2 \\
\end{array}
$$

\item
We create a TS$(6,4)$ with point set $\Z_5\cup\{\infty\}$ by developing the following blocks $\pmod 5$, where $\infty$ is a fixed point. A zero-sum $2$-flow is obtained by applying the given value of $f$ to each block developed from that starter block.
\[
\begin{array}{cccccc}
B & = & (\infty, 0,1) & (\infty, 0, 2) & (0,1,2) & (0,2,4)  \\
f(B) & = & 1 & -1 & 1 & -1 \\
\end{array}
\]

\item
Similarly, we create a TS$(6,6)$ with point set $\Z_5\cup\{\infty\}$ by developing the following blocks $\pmod 5$, where $\infty$ is a fixed point. A zero-sum $3$-flow is obtained by applying the given value of $f$ to each block developed from that starter block.
\[
\begin{array}{cccccccc}
B & = & (\infty, 0,1) & (\infty,0,1) & (\infty, 0, 2) & (0,1,2) & (0,2,4) & (0,2,4) \\
f(B) & = & 2 & -1 & -1 & 2 & -1 & -1 \\
\end{array}
\]
\end{enumerate}
\end{proof}

The existence of resolvable triple systems has been determined (see \cite{handbook}), giving us the following result.
\begin{theorem}
\label{resolvable ts}
A TS$(v,\lambda)$ with a zero-sum $3$-flow exists whenever $v\equiv 3 \pmod{6}$,  or $v\equiv 0\pmod 6$ and $\lambda$ is even, except when $(v,\lambda)\in\{(3,1), (6,2)\}$.
\end{theorem}
\begin{proof}
A resolvable TS$(v,\lambda)$ exists if and only if $v\equiv 3 \pmod{6}$, or $v\equiv 0\pmod 6$ and $\lambda$ is even, $v\neq 6$, \cite{handbook, col}. The result then follows from Theorem~\ref{resolvable}, noting that when $v=3$, Theorem~\ref{resolvable} does not apply as there is only one resolution class.
\end{proof}

Cyclic triple systems have also been well studied; their existence was established in \cite{col col}. In \cite{shalaby} the structure of the short orbits is considered. 

\begin{theorem}
\label{cyclic TS}
There exists a TS$(v,\lambda)$ which admits a zero-sum $3$-flow for the following:
\begin{enumerate}
\item
$\lambda \equiv 1 \pmod 6$ and $v \equiv 1 \pmod 6$;
\item
$\lambda \equiv 2, 10 \pmod{12}$ and $v \equiv 1, 4, 7 \pmod{12}$;
\item
$\lambda \equiv 3 \pmod{6}$ and $v \equiv 1 \pmod{2}$;
\item
$\lambda \equiv 4, 8 \pmod{12}$ and $v \equiv 1 \pmod{3}$;
\item
$\lambda \equiv 6 \pmod{12}$ and $v \equiv 0, 1, 3 \pmod{4}$;
\item
$\lambda \equiv 0 \pmod{12}$ and $v \geq 3$;
\end{enumerate}
\end{theorem}
\begin{proof}
For the given values of $v$ and $\lambda$, there exists a cyclic TS$(v,\lambda)$ with no short orbits \cite{col col,shalaby}. Apply Theorem~\ref{cyclic} to get the result.
\end{proof}

\begin{lemma}
\label{0mod6}
There exists a TS$(v,\lambda)$ with a zero-sum $3$-flow for every $\lambda\equiv 0\pmod 6$.
\end{lemma}
\begin{proof}
A TS$(6,6)$ with a zero-sum $3$-flow is given in Lemma~\ref{small}.
For $v > 6$, there exists an idempotent Latin square $L$ of order $v$ with an orthogonal mate $L'$. We take the symbol sets of these squares to be $\Z_v$. 
We form the design by taking the $v(v-1)$ triples 
$$\cB = \{(i, j, L(i,j)) \mid i, j\in \Z_v,\; i\neq j \}.$$

The fact that $L$ has an orthogonal mate, $L'$, means that $L$ can be decomposed into transversals, $T_i$, $i\in \Z_v$, each of which corresponds to the occurences of a given entry in $L'$ (see \cite{handbook}). We note that by reordering rows and columns and permuting symbols if necessary, we can arrange $L$ so that $T_{v-1}$ consists of the diagonal entries of $L$. Each $T_i$, $i\neq v-1$, then corresponds to a collection of blocks in the design. Taking the transversals in pairs, $T_i$, $T_{i+1}$, we label the blocks which come from $T_i$, $+1$, and $T_{i+1}$, $-1$; if $v$ is even (so there is an odd number of transversals), we label the blocks from $T_{0}$, $T_1$, $T_2$ with  $+2$, $-1$, $-1$ respectively.
\end{proof}

\begin{theorem}
\label{TS result}
There exists a TS$(v,\lambda)$ with a zero-sum $3$-flow for every admissible pair $(v,\lambda)$ except for $(v,\lambda)\in\{(3,1), (4,2), (6,2), (7,1)\}$ and except possibly when $v \equiv 10\pmod{12}$ and $\lambda = 2$.
\end{theorem}
\begin{proof}
The non-existence of any zero-sum flow for a TS$(3,1)$, TS$(4,2)$ and TS$(7,1)$ are given in Lemma~\ref{non} and the non-existence of a zero-sum $3$-flow for a TS$(6,2)$ is in Lemma~\ref{small}.
We note that for a fixed $v$ we may take copies of designs to build up $\lambda$. In particular, the existence of a TS$(v,6)$ from Lemma~\ref{small} means that often it suffices to consider only the first value $\pmod 6$ in $\lambda$. We consider the cases in $\lambda \pmod{6}$.
\begin{description}
\item{{\boldmath $\lambda\equiv 1,5\pmod 6$}} {\bf\boldmath ($v \equiv 1, 3 \pmod{6}$)}

For $v\equiv 1\pmod{6}$ the result follows from Theorem~\ref{cyclic TS}. 
For $v\equiv 3\pmod{6}$ the result follows from Theorem~\ref{resolvable ts}. 

\item{{\boldmath $\lambda\equiv 2,4\pmod 6$}} {\bf\boldmath ($v \equiv 0, 1 \pmod{3}$)}

When $v\equiv 0\pmod{3}$, $(v,\lambda)\neq (6,2)$, the result follows from Theorem~\ref{resolvable ts}.

When $v\equiv 1\pmod 3$ this is covered by Theorem~\ref{cyclic TS}, except when $v \equiv 10\pmod{12}$ and $\lambda \equiv 2,10\pmod{12}$. We now suppose $v \equiv 10\pmod{12}$ and consider the cases in $\lambda$.

When $\lambda \equiv 10\pmod{12}$, we may combine the blocks of a TS$(v,4)$ design with a zero-sum $3$-flow and a TS$(v,\lambda-4)$ with a zero-sum $3$-flow. Note that $\lambda-4\equiv 6\pmod{12}$, so this design exists by Lemma~\ref{0mod6}.

When $\lambda \equiv 2\pmod{12}$, $\lambda>2$, we may combine the blocks of a TS$(v,4)$ with a zero-sum $3$-flow and a TS$(v,\lambda-4)$ design with a zero-sum $3$-flow. Note that $\lambda-4\equiv 10\pmod{12}$, so this is the design above.

\item{\boldmath $\lambda\equiv 3\pmod 6$} {\bf\boldmath ($v \equiv 1 \pmod{2}$)}

This is covered by Theorem~\ref{cyclic TS}.

\item{\boldmath $\lambda\equiv 0\pmod 6$} {\bf\boldmath ($v \geq 3$)}

This is covered by  Lemma~\ref{0mod6}.
\end{description}
\end{proof}

Finally we give the following theorem which shows that in the case of $v \equiv 10\pmod{12}$ and $\lambda = 2$, there is a TS$(v,\lambda)$ with a zero-sum 5-flow.

\begin{theorem}
\label{2mod6 4}
There exists a TS$(v,\lambda)$ with a zero-sum 5-flow for every $v\equiv 4\pmod 6$, $v\neq 4$ and $\lambda\equiv 2,4\pmod 6$.
There exists a TS$(v,\lambda)$ with a zero-sum $3$-flow for every $v\equiv 1\pmod 6$, $v\neq 7, 19$ and $\lambda\equiv 2,4\pmod 6$.
\end{theorem}
\begin{proof}
We use a modification of the Bose construction; see \cite{col}. We construct the design on $(\Z_n\times \Z_3)\cup\{\infty\}$, where $n = (v-1)/3$, and denote points by $x_i$, where $x\in\Z_n$ and $i\in\Z_3$.
For $n\neq 2,6$ there exists an idempotent Latin square $L$ of side $n$ with an orthogonal mate $L'$~\cite{handbook}. 
For each pair $x,y\in\Z_n$, $x\neq y$, we form the design by taking the triples $(x_i, y_i, L(x,y)_{i+1})$.
We label the blocks which come from the transversals of $L$, $T_i$, $i<v-1$, with $+1$, $-1$ in pairs; if $v$ is odd, we label the blocks from $T_{0}$, $T_1$, $T_2$ with $+2, -1, -1$ respectively. This creates a partial design with a zero-sum $3$-flow on the points of $\Z_n\times \Z_3$.

For each vertical, $B_x = (x_0, x_1, x_2)$, $x\in \Z_n$, we place a TS$(4,2)$ on $B_x\cup\{\infty\}$ with the flow values as indicated below.
\[
\begin{array}{ccccc}
B = & (\infty, x_0, x_1) &  (\infty, x_1, x_2) & (\infty, x_0, x_2) & (x_0, x_1, x_2) \\
f_\alpha(B) = & \alpha & \alpha & \alpha  & -2 \alpha
\end{array}
\]
This gives weight 0 to each $x_i\in \Z_n\times \Z_3$ and weight $3 \alpha$ to $\infty$.
If $n = \frac{v-1}{3}$ is even we label the designs on the verticals with $\alpha = +1, -1$ in pairs, to get a zero-sum $3$-flow.
If $n$ is odd, we label three designs on the verticals with $\alpha = 1, 1, -2$ (which gives a block of weight 4) and the rest with $\alpha = +1, -1$  in pairs to get a zero-sum 5-flow ($n > 2$).
\end{proof}

\section{An $O((\lambda v)^2)$ Bound for Zero-Sum Flows}
\label{bound}

In this section we establish an $O((\lambda v)^2)$ bound on the size of a zero-sum flow on a TS$(v,\lambda)$ when $v>4$, $v\neq 7$. Further, this zero-sum $O((\lambda v)^2)$-flow takes at most five distinct values.

\begin{theorem}
For $v>4$, $v\neq 7$, every TS$(v, \lambda)$ admits a zero-sum flow whose entries are in the set 
\[
S = \left\{ \;
\frac{-\lambda(v-3)}{2} \left(\frac{\lambda(v-3)}{2}-2 \right), \; -3\lambda, \; 
\frac{3\lambda(v-3)}{2},\; \frac{\lambda(v-7)}{2},\; \lambda(v-4)\;  
  \right\}.
\]
\end{theorem}
\begin{proof}
Consider a TS$(v, \lambda)$ with incidence matrix $N$. We know that the number of appearances of each element in  blocks is $r=\frac{\lambda(v-1)}{2}$. 
We show that the last column of $N$ is a linear combination of the other columns with coefficients from the set $3S$. 

By a suitable ordering of the elements of $\{1, \dots ,v\}$, we can assume that the last column of $N$ is $Z=[1, 1 , 1, 0,\ldots ,0]^T$.
Now, we remove the last column of $N$ and call the remaining matrix $M$. We then have the following equality:
\[
L=MM^T=\left[
 \begin{array}{ccccccccccc}
   (r-\lambda)I_3 +(\lambda-1)J_3 & & &  \lambda J_{3,v-3}\\
    \lambda J_{v-3, 3} & & & (r-\lambda)I_{v-3}+ \lambda J_{v-3} \\
  \end{array}\right], 
\]
where $J_{p,q}$ is the $p \times q$ all-1 matrix.  For simplicity we denote $J_{p,p}$ by $J_p$.

We now show that the system of linear equations $LY=tZ$ has integral solutions such that each component of $Y$ is in the set $\{-3\lambda, \frac{3\lambda(v-3)}{2}\}$.
Assume that $Y=[a,a,a,b,\ldots ,b]^T$, where $a$ and $b$ are unknown variables and there are $v-3$ $b$s. 
We may solve the equations
\[
\begin{array}{ccc}
(r-\lambda)a+2(\lambda-1)a+ \lambda(v-3)b=t & \mbox{ and } & 3\lambda a+(r+\lambda(v-4))b=0,
\end{array}
\]
to get 
$a=r-\lambda(v-4)= \frac{3\lambda(v-3)}{2}$, $b=-3\lambda\;$ ($v>4$) and
$$t=(r+\lambda-2)(r-\lambda(v-4))-3\lambda^2(v-3)=\frac{-3\lambda(v-3)}{2} \left(\frac{\lambda(v-3)}{2}-2 \right).$$
Now, since $MM^TY=tZ$ and each row of $M^T$ has exactly three $1$s, we may conclude that there exists a vector, $X$, in the nullspace of $N$ whose components are all of the form $-t$, $3a$, $2a+b$, $a+2b$ or $3b$.
Finally, we note that when $v>4$, $v\neq 7$, all of the components of $X$ are non-zero integers which are divisible by 3. Now, $X'=\frac{1}{3}X$ is also in the nullspace of $N$ and all of its components are in the set $S$.
\end{proof}

\section{\boldmath$(2v+1)$-construction for zero-sum flows on STS($v$)}
\label{recursive}

In this section we show that the standard $(2v+1)$-construction for Steiner triple systems can be adapted to respect zero-sum flows in many cases. While we work in the case $\lambda=1$, the generalisation to higher $\lambda$ is clear.

We say that a graph $G$ has a {\em $k$-null 1-factorisation} if $G$ has a zero-sum $k$-flow and there is a $1$-factorisation in which the weight of each $1$-factor is zero. We call each 1-factor of a $k$-null 1-factorisation a {\em $k$-null 1-factor}.

\begin{lemma}
\label{2v+1}
If $K_{v+1}$ has a $k$-null 1-factorisation and there exists an STS$(v)$, $\cS$, with a zero-sum $\ell$-flow, then $\cS$ can be embedded into an STS$(2v+1)$ with a zero-sum $\max(k,\ell)$-flow.
\end{lemma}
\begin{proof}
Let $(X,\cB)$ be an STS$(v)$, with $X=\{x_1, \ldots, x_v\}$, and let $Y$ be a set of order $v+1$. We will construct the new design on $X\cup Y$. Construct a $k$-null 1-factorisation on $K_{v+1}$ with point set $Y$. Let $F_1, \ldots , F_{v}$ be  the 1-factors in this $1$-factorisation, and suppose that $F_i=\{\{y_{ij}, z_{ij}\}\mid 1\leq j\leq \frac{v+1}{2}\,\},$ for $i=1, \ldots, v$.
Form the triples 
$$
\cC = \bigcup_{1\leq i\leq v,\, 1\leq j \leq \frac{v+1}{2}} \{x_i, y_{ij}, z_{ij}\},
$$
it is easy to see that $\cB\cup \cC$ is an STS$(2v+1)$. In order to obtain a zero-sum $\max(k,\ell)$-flow, we retain the original weighting on $\cB$, and for each triple $\{x_i, y_{ij}, z_{ij}\}\in \cC$, we give it the weight of the edge $\{y_{ij}, z_{ij}\}$ in $F_i$. It is not hard to see that we obtain a zero-sum $\max(k,\ell)$-flow on the STS$(2v+1)$, $(X\cup Y,\cB\cup \cC)$.
\end{proof}

\begin{lemma}
\label{bipart} 
There exists a $3$-null 1-factorisation of $K_{n,n}$ for every $n\geq 3$.
\end{lemma}
\begin{proof}
First assume that $n\neq 6$ and let $L_1$ and $L_2$ be two orthogonal Latin squares of order $n$.
Let $U=\{u_1,\ldots ,u_n\}$ and $V=\{v_1,\ldots ,v_n\}$ be the two parts of the complete bipartite graph $K_{n,n}$. We use $L_1$ to define a colouring, $c$, on the edges of $K_{n,n}$, and $L_2$ to factor $K_{n,n}$ into $k$-null 1-factors as follows. For each edge $u_iv_j\in E(K_{n,n})$, $1\leq i,j\leq n$, we give $u_iv_j$ colour $c(u_i,v_j)=L_1(i,j)$. We now define a 1-factorisation of $K_{n,n}$ by, $F_k = \{(u_i,v_j) \mid L_2(i,j) = k\}$. Note that since $L_1$ and $L_2$ are orthogonal, each $F_i$ contains exactly one edge of each colour.
If $n$ is even, then we assign $1$ to every edge whose colour is in the set $\{1,\ldots, \frac{n}{2}\}$ and assign $-1$ to the remaining edges. If $n$ is odd, then we assign $2, -1,-1$ to all edges with colours, $1,2$ and $3$, respectively and assign $1,-1$ to all edges with colours $4,5,\ldots, n$, alternately. It is not hard to see that this assignment is the desired zero-sum $3$-flow for $K_{n,n}$.

For the case $n=6$, the edges of $K_{6,6}$ can be decomposed into four subgraphs isomorphic to $K_{3,3}$. Since every $K_{3,3}$ has the desired zero-sum $3$-flow, we are done.
\end{proof}

\begin{lemma}
\label{4} 
There exists a $3$-null 1-factorisation of $K_{n}$ for every $n\equiv 0\pmod 4$, $n \neq 4$.
\end{lemma}
\begin{proof}
Let $n=4r$, $r>1$, and consider the complete graph $K_{n}$ as the join of two complete graphs $K_{2r}$ and $K_{2r}$. By Lemma~\ref{bipart}, $K_{2r,2r}$ has a $3$-null 1-factorisation.
Let $M_1,\ldots, M_{2r-1}$ and $M'_1,\ldots, M'_{2r-1}$ be two $1$-factorisations for the first and the second $K_{2r}$, respectively. Then $\{M_i\cup M'_i \mid 1\leq i\leq 2r-1\}$ forms a $1$-factorisation for the disjoint union of the two $K_{2r}$.
Now assign $2$ to all edges of $M_1$, $-2$ to all edges of $M'_1$, $-1$ to all edges of $M_2$ and $1$  to all edges of $M'_2$. For each $i$, $3\leq i\leq 2r-1$ assign $-1$ and $1$ to all edges of $M_i$, alternately. For each $i$, $3\leq i\leq 2r-1$ assign $1$ and $-1$ to $M'_i$, alternately. 
\end{proof}

Now, using Lemmas \ref{2v+1} and \ref{4}, we have the following result.

\begin{theorem}
\label{recursive 3-flow}
If $v \equiv 3\pmod 4$, $v>3$, then every ${\rm STS}(v)$ with a zero-sum $3$-flow can be embedded into an ${\rm STS}(2v+1)$ with a zero-sum $3$-flow.
\end{theorem}

We note that the requirement $v \equiv 3\pmod 4$ and the existence of an STS$(v)$ implies that $v\equiv 3, 7\pmod{12}$.

A similar proof to that of Lemma \ref{4} shows that if $n\equiv 0\pmod 8$ then $K_n$ has a zero-sum $2$-flow with a $1$-factorisation in which the weight of each $1$-factor is zero.
By Lemma~\ref{2v+1}, this shows that when $v\equiv 7\pmod 8$ every ${\rm STS}(v)$ with a zero-sum $2$-flow can be embedded into an ${\rm STS}(2v+1)$ with a zero-sum $2$-flow.

\begin{lemma}
\label{null 6k+4}
There exists a $4$-null 1-factorisation of $K_{6k+4}$ for every $k>1$.
\end{lemma}
\begin{proof} 
Let the point set of $K_{6k+4}$ be $\Z_{6k+3}\cup\{\infty\}$ and let $F_1, \ldots , F_{6k+3}$ be the $1$-factorisation of $K_{6k+4}$ defined by $F_i=F_1+i$, $0\leq i\leq 6k+2$, where $F_1=\{(x, -x)\mid x\in
\{1,\ldots, 3k+1\} \} \cup \{(0, \infty)\}$. We claim that $K_{6k+4}$ can be partitioned into $2k+1$ cubic graphs each isomorphic to 
$$K_{3,3}\cup K_{3,3}\cup \cdots \cup K_{3,3} \cup K_4,$$ 
where the number of $K_{3,3}$ is $k$.

It is not hard to see that $F_1\cup F_{2k+2}\cup F_{4k+3}$ is a disjoint union of a $K_4$ with vertex set $\{\pm (2k+1), 0, \infty \}$ and $k$ copies of $K_{3,3}$ with the vertex sets 
\[
\{\pm i, \pm (2k+1-i), \pm (2k+1+i)\},\; 1\leq i\leq k,
\]
and partite sets 
\[
X_i=\{i, -(2k+1-i), 2k+1+i\} \mbox{ and } Y_i=\{-i, 2k+1-i, -(2k+1+i)\}.
\]
Clearly, $F_i\cup F_{2k+1+i}\cup F_{4k+2+i}\simeq F_1\cup F_{2k+2}\cup F_{4k+3}$, for $i=1, \ldots , 2k+1$ and so it can also be decomposed as required.

Now, consider the following edge assignment for disjoint union of two $K_{3,3}$ and a $K_4$.
\begin{center}
\includegraphics[height=80pt]{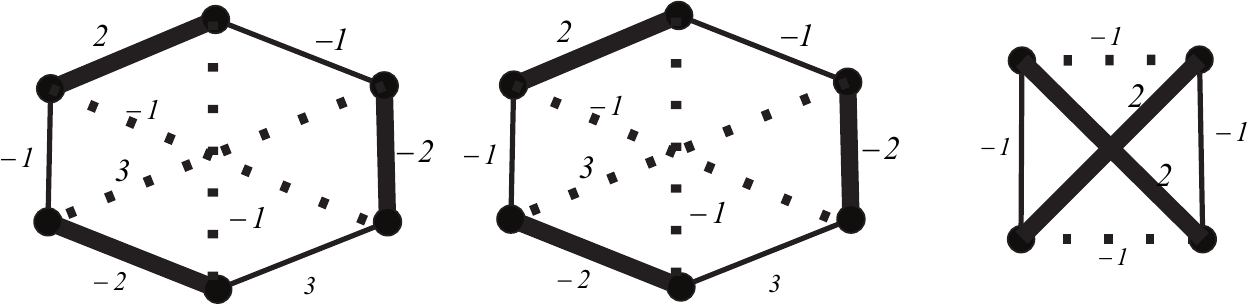}
\end{center}

Thus, by this assignment, $K_{3,3}\cup K_{3,3}\cup K_4$ admits a $4$-null 1-factorisation.
On the other hand, by Lemma \ref{bipart}, $K_{3,3}$ admits a $4$-null 1-factorisation.
\end{proof}

We may use  Lemma~\ref{null 6k+4} above in Lemma~\ref{2v+1} to get the following result.

\begin{theorem} 
\label{recursive 4-flow}
If $v\equiv 3 \pmod 6$, $v>9$, then every ${\rm STS}(v)$ with a zero-sum $4$-flow can be embedded in an ${\rm STS}(2v+1)$ with a zero-sum $4$-flow.
\end{theorem}

Given two graphs $G$ and $H$, the {\it Cartesian product} of $G$ and $H$, $G\square H$, is the graph with vertex set $V(G)\times V(H)$, and $(x,a)\, (y,b) \in E(G\square H)$ if and only if either $x=y$ and $a\,b\in E(H)$, or $a=b$ and $x\,y\in E(G)$.
We can deal with the case $v=9$ allowing for a $5$-flow.
\begin{lemma}
\label{K_10}
There exists a $5$-null 1-factorisation of $K_{10}$.
\end{lemma}
\begin{proof}
It is not hard to see that $K_{10}$ can be decomposed into $3$ cubic graphs 
\[
C_5\square K_2, C_5\square K_2, (K_3\square K_2)\cup K_4.
\]
Now, the following edge assignments imply that there is a $5$-null 1-factorisation of $K_{10}$.
\begin{center}
\includegraphics[height=80pt]{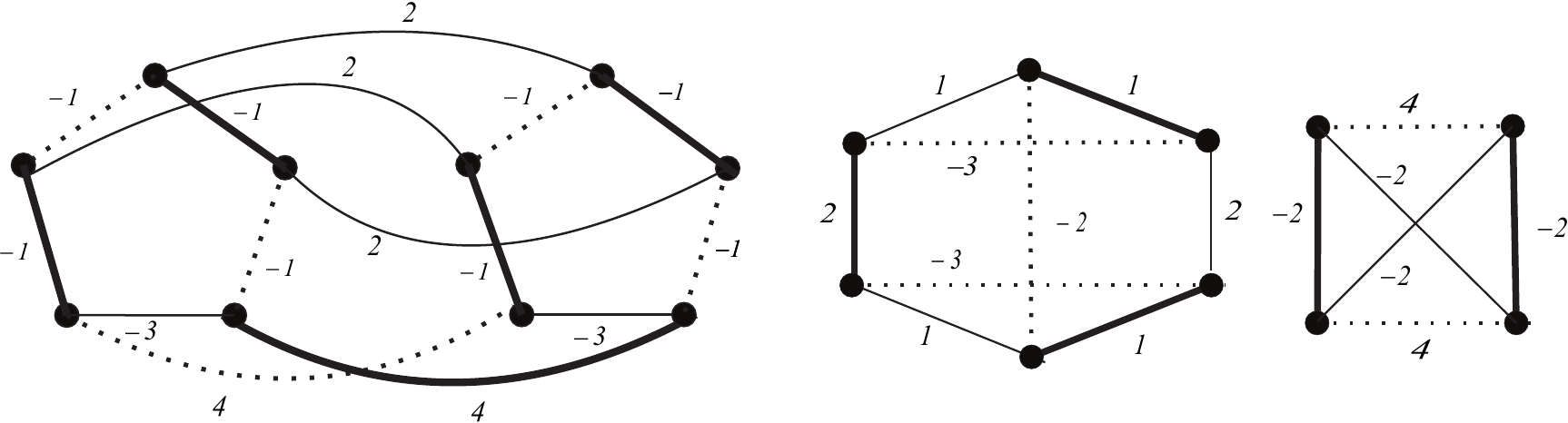}
\end{center}
\end{proof}
Given two graphs $G$ and $H$, the {\it wreath product} of $G$ and $H$, $G\wr H$, is the graph with vertex set $V(G)\times V(H)$, and $(x,a)\, (y,b) \in E(G\wr H)$ if and only if $x\,y\in E(G)$, or $x=y$ and $a\,b\in E(H)$.

A {\it $2$-factor} is a collection of cycles that spans all vertices of the graph. A {\it $2$-factorisation} of a graph $G$ is an edge decomposition of $G$ into $2$-factors. 
The problem of finding an $F$-factorisation of $K_n$, where $F$ is a given $2$-factor of order $n$, is the well known Oberwolfach problem. Clearly a $2$-factorisation of $K_n$ cannot exist when $n$ is even; in this case it is common practice to consider factorisations of $K_n - I$, the complete graph with the edges of a 1-factor $I$ removed. It is known that the Oberwolfach problem has no solution when $F \in \{C_3\cup C_3, C_3 \cup C_3 \cup C_3 \cup C_3, C_4\cup C_5, C_3\cup C_3\cup C_5\}$. Otherwise, a solution is known for every case where $n\leq 40$ \cite{DFWMR 10}, and if every component of the factor is isomorphic \cite{als, ASSW 89, Hoffman Schellenberg 91}. The case where $F$ is bipartite is nearly solved \cite{BD, hag}; the case where $F$ consists of exactly two parts  is solved~\cite{Bryant 01, Traetta 13}. Rotational solutions have been studied \cite{Buratti Rinaldi 05, Buratti Rinaldi 08, Buratti Traetta 12} and many other families are known \cite{Bryant Schar 09, Hilton Johnson 01, Ollis 05}, but no general solution is known. See \cite[Section VI.12]{handbook} for a survey.
H\"{a}ggkvist proved the following very useful result in \cite{hag}.

\begin{lemma}[\cite{hag}]
\label{hag}
For any $m > 1$ and for each bipartite $2$-regular graph $F$ of order $2m$, there exists a $2$-factorisation of $C_m\wr \overline{K_2}$, in which each $2$-factor is isomorphic to $F$.
\end{lemma}

Since it is well known that $K_{2k+1}$ has a Hamiltonian factorisation, we obtain the following theorem, which was proved in \cite{als}.

\begin{theorem}[\cite{als}]
\label{wr}
For every positive integer $k$, the graph $K_{4k+2}$ can be decomposed into $k-1$, graphs isomorphic to $C_{2k+1}\wr \overline{K_2}$, and one graph isomorphic to $C_{2k+1}\wr K_2$.
\end{theorem}

\begin{lemma}
\label{null 4k+2}
There exists a $5$-null 1-factorisation of $K_{4k+2}$ for every $k>1$.
\end{lemma}
\begin{proof} 
By Theorem \ref{wr}, $K_{4k+2}$ can be decomposed into $k-1$, graphs isomorphic to $C_{2k+1}\wr \overline{K_2}$, and one $C_{2k+1}\wr K_2$.
Now, let $F$ be a $2$-regular graph which is disjoint union of one $C_6$ and $k-1$, $C_4$s. 
\begin{figure}[h]
\caption{\label{H_i}
The cubic graph $H_i$}
\begin{center}
\includegraphics[height=70pt]{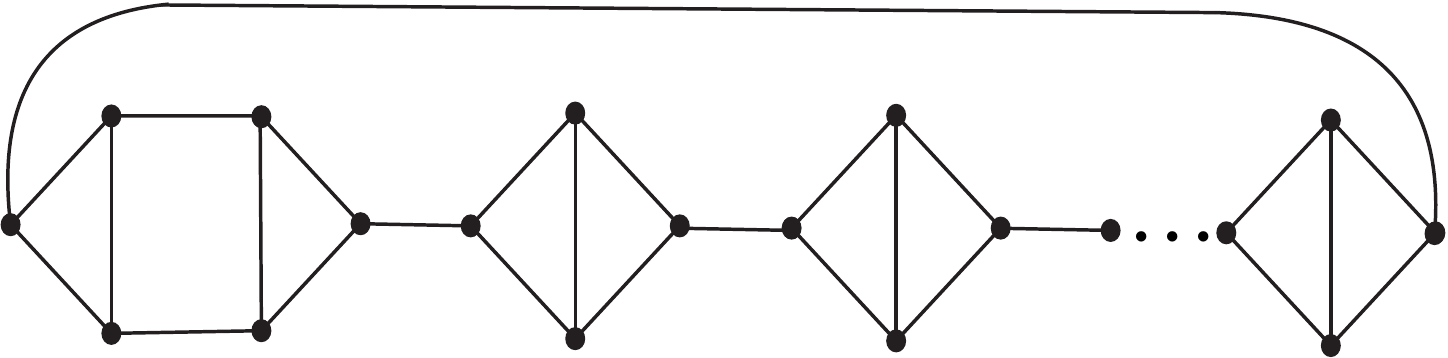}
\end{center}
\end{figure}
By Lemma \ref{hag}, $C_{2k+1}\wr \overline{K_2}$, has a $2$-factorisation in which each $2$-factor is isomorphic to $F$. Assign $-2$ and $2$ to the edges of $C_6$, alternately. If the number of
$C_4$ in $F$ is even, then assign $-2$ and $2$ to the edges of one $C_4$, alternately and assign $-1$ and $1$ to the edges of each other $C_4$, alternately. It is not hard to see that $F$ has a $3$-null 1-factorisation.
If the number of $C_4$ is odd, then assign  $-3$ and $3$ to the edges of one $C_4$, alternately and assign $-1$ and $1$ to the edges of each other $C_4$, alternately.
Again, it is not hard to see that $F$ has a $3$-null 1-factorisation.
\begin{figure}[h]
\caption{\label{decomp}
The edge decomposition of  $C_{2k+1}\wr K_2$ into $F$ and $H_{k-1}$}
\begin{center}
\includegraphics[height=70pt]{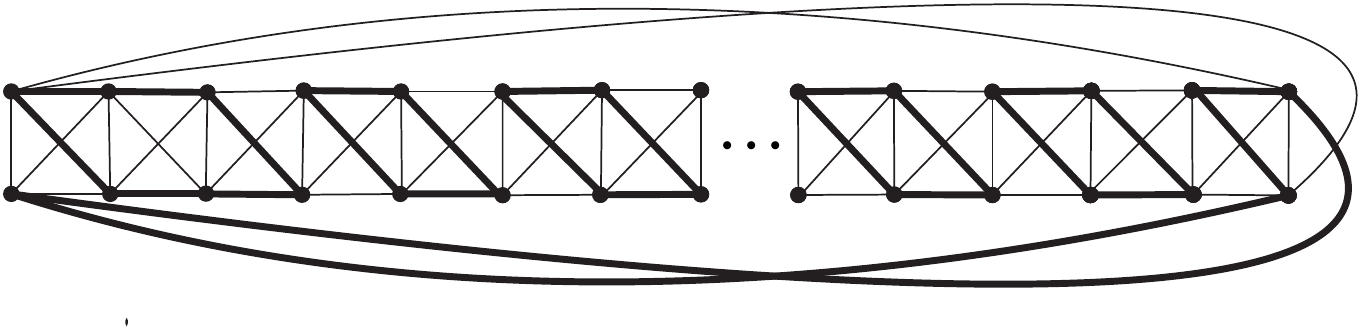}
\end{center}
\end{figure}

\begin{figure}
\caption{\label{5-flow H_1}
5-flow on $H_1$}
\begin{center}
\includegraphics[height=70pt]{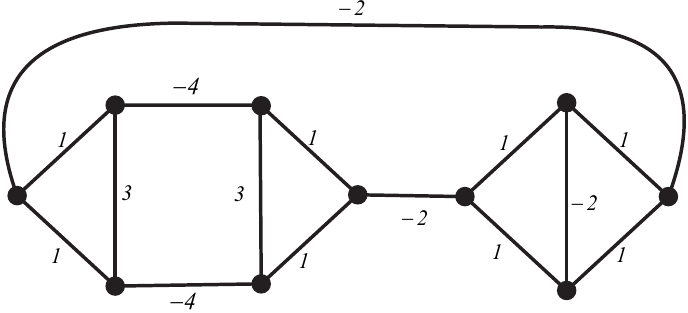} \\
\end{center}
\end{figure}

\begin{figure}
\caption{\label{5-flow H_3}
5-flow on $H_3$}
\begin{center}
\includegraphics[height=70pt]{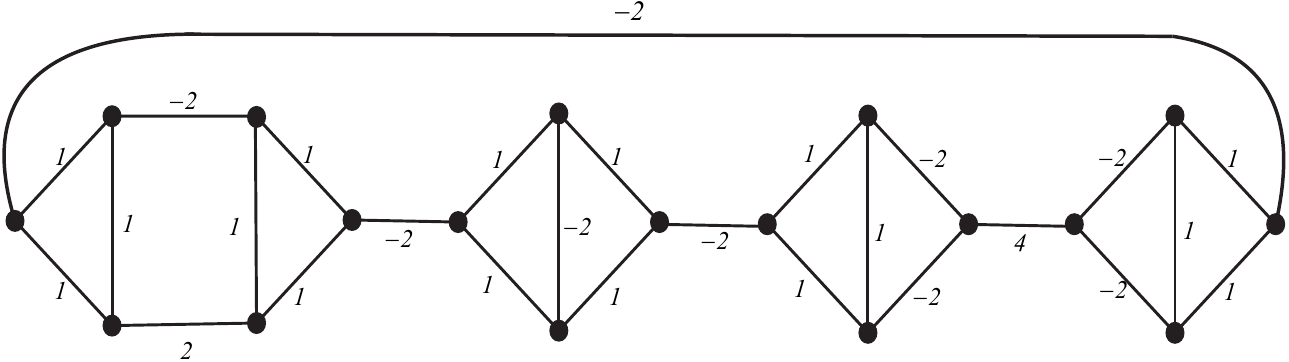} \\
\end{center}
\end{figure}

\begin{figure}
\caption{\label{5-flow H_5}
5-flow on $H_5$}
\begin{center}
\includegraphics[height=70pt]{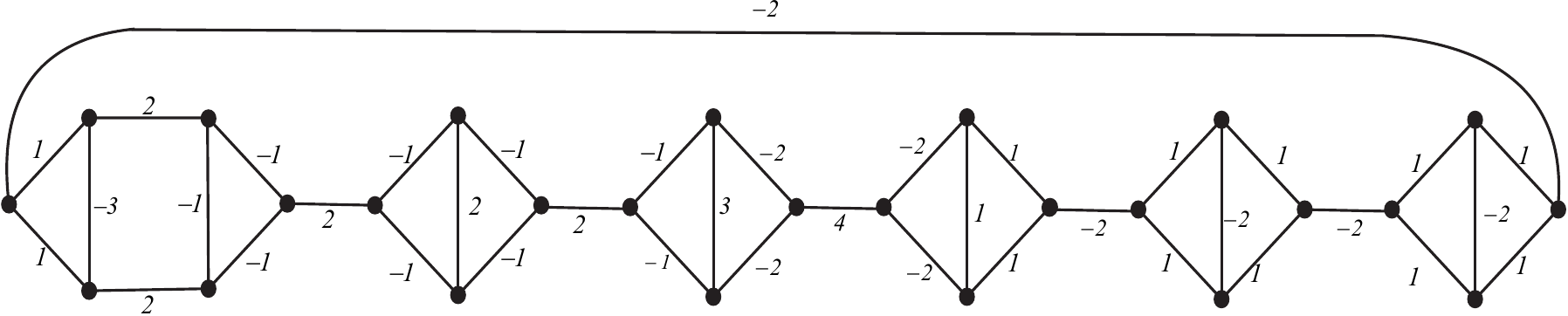} \\
\end{center}
\end{figure}
\comment{
\begin{figure}
\caption{\label{5-flow}
5-flow on $H_1$, $H_3$ and $H_5$}
\begin{center}
\begin{tabular}{c}
\includegraphics[height=70pt]{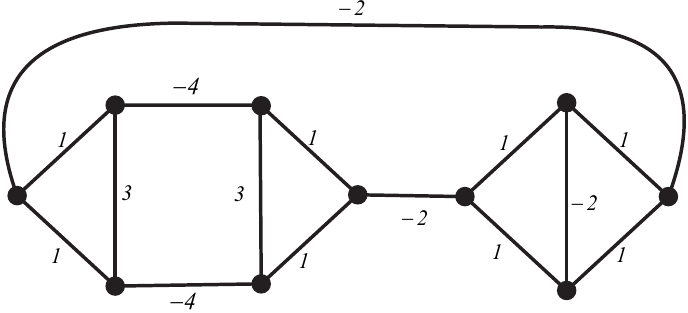} \\
\includegraphics[height=70pt]{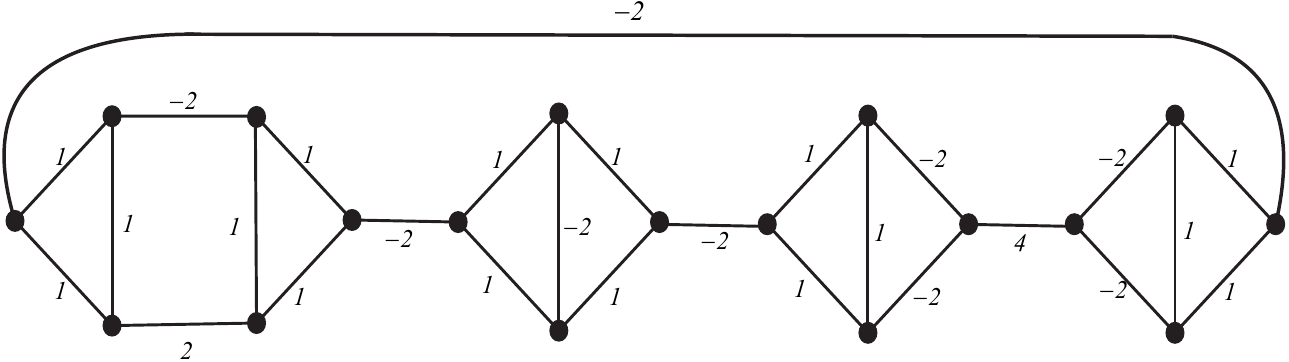} \\
\includegraphics[height=70pt]{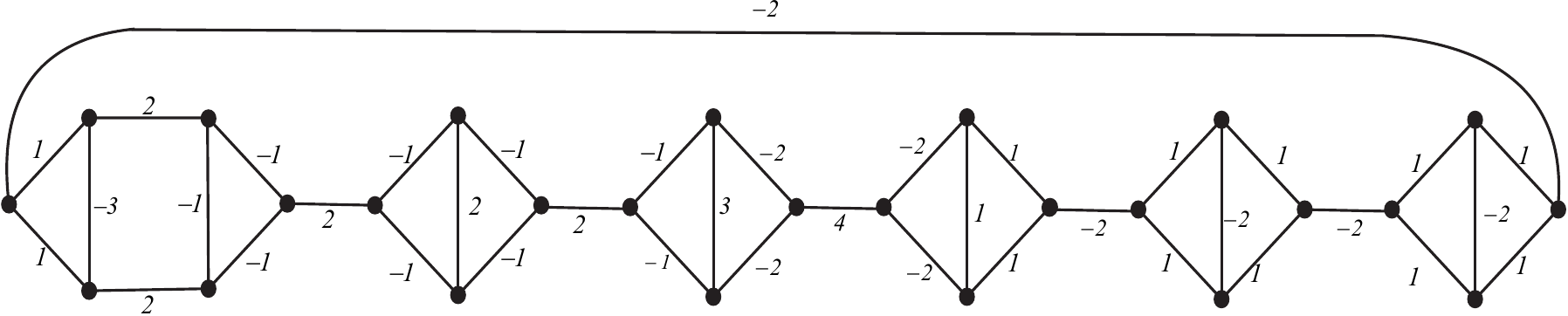} \\
\end{tabular}
\end{center}
\end{figure}
}

Let $H_i$ be the cubic graph shown in Figure~\ref{H_i}, where there are exactly $i$ blocks of $K_4$ minus one edge.
Now, consider the $5$-regular graph $C_{2k+1}\wr K_2$, Figure~\ref{decomp} shows that $C_{2k+1}\wr K_2$ can be decomposed into a copy of $F$ and the cubic graph $H_{k-1}$.
Figures~\ref{5-flow H_1}, \ref{5-flow H_3} and \ref{5-flow H_5} show a $5$-null 1-facorisation of $H_1$, $H_3$ and $H_5$, respectively.
Figure~\ref{continue} gives a graph which can be appended to $H_1$, $H_3$ or $H_5$ to get a zero-sum $5$-flow with a $1$-factor in which the weight of each $1$-factor is zero for any $H_i$, except $i=2$.

\begin{figure}
\caption{\label{continue}
Continuation Graph}
\begin{center}
\includegraphics[height=50pt]{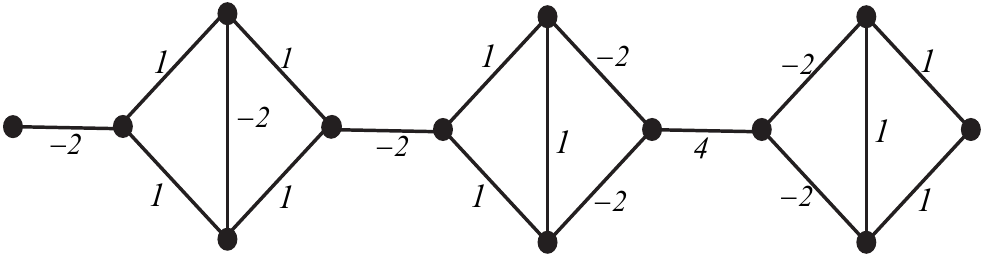}
\end{center}
\end{figure}

Now, for the case $i=2$, so $k=3$, let $L$ be the bipartite $2$-factor which is the disjoint union of a $C_6$ and a $C_8$. 
It is not hard to see that $C_7\wr K_2$ can be decomposed into $L$ and the cubic graph given in Figure~\ref{k=3}, which has a zero-sum $5$-flow with a desired $1$-factor.

\begin{figure}
\caption{\label{k=3}
Zero-sum $5$-flow on $L$}
\begin{center}
\includegraphics[height=70pt]{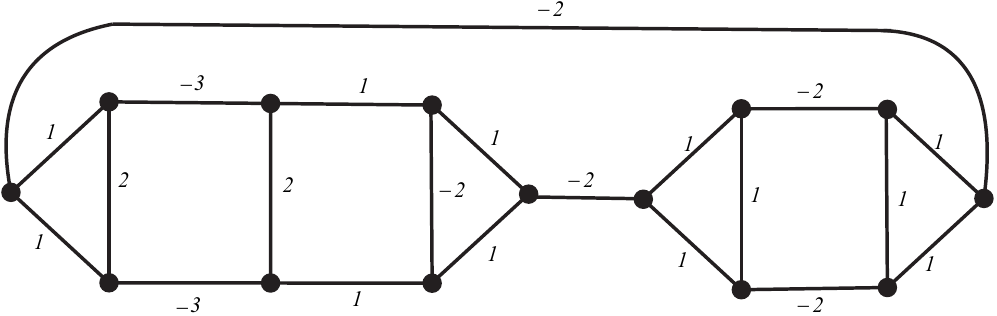}
\end{center}
\end{figure}

Now, alternately give weight $3$ and $-3$ to all edges of $C_8$ and $4$ and $-4$ to all edges of $C_6$ in $L$, alternately. This implies that $K_{14}$ has a zero-sum $5$-flow with a $1$-factorisation in which the weight of each $1$-factor is zero. 
\end{proof}

Now, using Lemma~\ref{null 4k+2} in Lemma~\ref{2v+1} we get the following result.
\begin{theorem}
\label{recursive 5-flow}
If $v\equiv 1 \pmod 4$, then every STS$(v)$ with a zero-sum $5$-flow can be embedded in an STS$(2v+1)$ with a zero-sum $5$-flow.
\end{theorem}

Noting that the necessary conditions for the existence of an STS$(v)$ are $v\equiv 1,3,7,9\pmod{12}$, we may summarise the results of this section in the following theorem.
\begin{theorem}
An STS$(v)$ admitting a zero-sum $k$-flow may be embedded in an STS$(2v+1)$ which admits a zero-sum $k$-flow under the following conditions:
\begin{itemize}
\item
$k\geq 3$, $v\equiv 3, 7\pmod{12}$, $v>7$;
\item
$k\geq 4$, $v\equiv 9\pmod{12}$, $v>9$;
\item
$k\geq 5$, $v\equiv 1 \pmod{12}$, or $v=9$.
\end{itemize}
\end{theorem}
We note that the zero-sum flow on the blocks of the embedded STS$(v)$ remains unchanged.
\\[2ex]

\noindent
{\bf\large Acknowledgments} \\ 
This work was done while S.Akbari was visiting the University of Toronto and supported by NSERC Discovery Grant 455994. Also, the research of the first author was in part supported by a grant from IPM (No.\ 93050212). The second and third authors are both supported by the NSERC Discovery Grant program.

\end{document}